\documentclass[reqno, 12pt]{article}

\pdfoutput=1

\usepackage{enumerate}
\usepackage{latexsym}
\usepackage[centertags]{amsmath}
\usepackage{amsfonts}
\usepackage{amssymb}
\usepackage{amsthm}
\usepackage{newlfont}
\usepackage{graphics}
\usepackage{color}
\usepackage{float}
\usepackage{diagbox}
\usepackage{hyperref}
\usepackage{longtable}
\usepackage{rotating}
\usepackage{multirow}
\usepackage{extarrows}
\usepackage[sort,compress,numbers]{natbib}
\usepackage[utf8]{inputenc}

\numberwithin{equation}{subsection}
\newtheorem{thm}{Theorem}[section]
\newtheorem{cor}[thm]{Corollary}
\newtheorem{lem}[thm]{Lemma}
\newtheorem{prop}[thm]{Proposition}

\theoremstyle{mydefinition}
\newtheorem{dfn}[thm]{Definition}
\theoremstyle{myremark}

\allowdisplaybreaks[4]

\DeclareMathOperator{\Cot}{Cot}
\DeclareMathOperator{\Tan}{Tan}
\DeclareMathOperator{\Csc}{Csc}
\DeclareMathOperator{\Sec}{Sec}

\def\CT{\mathop{\mathrm{CT}}}

\def\Res{\mathop{\mathrm{Res}}}

\title{Root-of-unity weighted trigonometric power sums: a constant term approach}

\author{Jiahang Liu$^{1}$ and Guoce Xin$^{2,*}$\\
\small $^{1,2}$School of Mathematical Sciences, Capital Normal University, Beijing 100048, P.R.\ China\\
\small $^{1}$\texttt{jiahang-liu@foxmail.com}; $^{2}$\texttt{guoce\_xin@163.com}}

\date{\today}

\begin{document}
\maketitle

\begin{abstract}
We present a unified method for evaluating finite sums of powers of cotangent, tangent,
cosecant and secant weighted by primitive $k$th roots of unity.
The approach relies on constant term extraction for iterated Laurent series, combined
with generating functions and partial fraction decomposition.
We obtain explicit closed-form expressions for all even-power sums of the four functions,
as well as for odd-power cotangent and tangent sums.
The formulas are given in terms of Bernoulli polynomials, Euler polynomials,
and universal coefficients $r_{n,t}$.
As applications, we recover numerous classical identities---including the ordinary and
alternating cotangent power sums and Acton's alternating tangent sum---in a systematic
and elementary way.
\end{abstract}

\medskip
\noindent\textbf{Keywords:}  constant term method; Laurent series; trigonometric sum; Bernoulli polynomial; Euler polynomial.

\noindent
\begin{small}
 \emph{Mathematics subject classification}: Primary 11L03; Secondary 05A15.
\end{small}

%

\section{Introduction}
\label{sec:intro}

Finite power sums of trigonometric functions weighted by roots of unity lie at the intersection
of analytic number theory, combinatorics and the theory of modular forms.  They appear in the
study of Dedekind sums, special values of Dirichlet $L$-functions, and lattice sums, and have
been investigated by many authors \cite{Berndt2002,Cvijovic2009,franke2021}.

Classical approaches to such sums often rely on contour integration, generating functions or
properties of fractional part functions.  For instance, Girstmair~\cite{girstmair1987} expressed
cotangent sums via fractional parts and used them to evaluate Dirichlet series, while Bettin and
Conrey~\cite{BettinConrey2013} employed harmonic analysis and recurrence relations.  Cvijovi{\'c}
\cite{Cvijovic2012} systematically applied partial fraction decompositions of trigonometric
functions to derive a dozen explicit identities, but his method required separate analyses for
each parity and involved cumbersome manipulations of poles.  Acton~\cite{AMM4632} posed a
classical alternating tangent sum that was originally solved by contour integration; more
recently, Dowker~\cite{Dowker1987} introduced a class of ``twisted'' trigonometric sums with
roots of unity that motivate the present work.

A common feature of many existing closed forms is their reliance on nested convolutions or
iterated sums, which often obscure the structural simplicity of the results.  Moreover, while
even-power sums of the four basic trigonometric functions (cotangent, tangent, cosecant, secant)
can be handled by elementary identities, their odd-power counterparts are far less understood.

In this paper we propose a unified \emph{constant term} approach that provides explicit,
conceptually transparent closed forms for all even-power sums and for odd-power sums of cotangent
and tangent, weighted by primitive $k$-th roots of unity $\zeta_k=e^{2\pi i/k}$.  Our starting
point is the introduction of a real parameter $\theta$ that allows us to treat singularities in a
natural way:
\[
\Cot(a,n,k;\theta)=\sum_{\substack{j=0\\ j\neq j_0}}^{k-1}\zeta_k^{aj}\cot^{\,n}\Bigl(\theta+\frac{\pi j}{k}\Bigr),
\]
and analogously for tangent, cosecant and secant sums ($\Tan$, $\Csc$, $\Sec$).  The additional
index $j_0$ marks the unique pole of the summand, which exists for certain values of $\theta$
and is essential for a clean description when $\theta=0$.

Our main tool is the constant term operator on iterated Laurent series, a technique developed by
the first named author in \cite{Xin2004} and further refined in \cite{Xin2015euclid}.  Instead of
working with trigonometric identities, we use Euler's formula to convert the sum into a rational
function evaluation and then systematically extract the constant term via a change of variables
and partial fraction decomposition (Propositions~\ref{prop:change-var} and~\ref{prop:Cot-CT}).
This method not only circumvents the need for separate parity-based arguments but also reveals a
direct link with Bernoulli and Euler polynomials.

The main results of the paper are as follows:
\begin{itemize}
    \item \emph{Cotangent sums.}  For every integer $n\ge1$, the even- and odd-power sums
    $\Cot(a,2n,k)$ and $\Cot(a,2n+1,k)$ are expressed in terms of Bernoulli polynomials
    $B_m(a/k)$ and the universal coefficients $r_{n,t}$ (Theorems~\ref{thm:cot-even}
    and~\ref{thm:cot-odd}).  The cosine- and sine-weighted sums follow immediately as
    corollaries.
    \item \emph{Tangent sums.}  The formulas for $\Tan(a,n,k)$ exhibit a parity dependence on
    $k$.  When $k$ is even, $\Tan$ reduces directly to a cotangent sum; when $k$ is odd, it
    involves Euler polynomials $E_m(a/k)$ (Theorem~\ref{thm:tan}).  This provides the first
    complete description of root-of-unity weighted tangent power sums.
    \item \emph{Cosecant and secant sums.}  For even powers, the identities $\csc^2x=1+\cot^2x$
    and $\sec^2x=1+\tan^2x$ together with the binomial theorem give
    $\Csc(a,2n,k;\theta)$ and $\Sec(a,2n,k;\theta)$ as explicit linear combinations of the
    corresponding cotangent and tangent sums (Theorems~\ref{thm:csc} and~\ref{thm:sec}).
\end{itemize}

The generality of our formulas is demonstrated by a large collection of explicit corollaries: we
recover, often in a more compact form, classical results such as the ordinary and alternating
cotangent sums, the weighted cosecant square sum, and Acton's alternating tangent sum
\cite{AMM4632}.  Many of these identities were previously obtained by different ad hoc methods;
our constant term framework unifies them under a single algebraic principle.

The rest of the paper is organized as follows.  Section~\ref{sec:prelim} collects the necessary preliminaries
on Bernoulli and Euler polynomials, the coefficients $r_{n,t}$, and the constant term operators.
In Section~\ref{sec:cot} we prove the closed forms for weighted cotangent power sums.
Section~\ref{sec:apps} extends the method to tangent, cosecant and secant sums, and presents
numerous applications.


\section{Preliminaries}
\label{sec:prelim}

\subsection{Bernoulli polynomials and related numbers}
\label{subsec:Bernoulli}

The Bernoulli polynomials of order $r$ are defined by the generating function
\begin{equation}\label{e-Ber-Poly-r}
       \left(\frac{x}{e^x - 1}\right)^{\!\!r} e^{tx} = \sum_{m \ge 0} B_m^{(r)}(t) \frac{x^m}{m!}.
\end{equation}
When $r=1$, this reduces to the ordinary Bernoulli polynomials $B_m(x)$, and when $x$ is also taken as $0$,
it gives the Bernoulli numbers $B_m = B_m(0)$, i.e.,
\begin{equation}\label{e-Ber-Num}
\frac{x}{e^x-1}=\sum_{m\geq0}B_m \frac{x^m}{m!}.
\end{equation}
All odd Bernoulli numbers with $m \ge 3$ vanish identically.

For our later computations we will frequently use the one-parameter generating function
\begin{equation}\label{e-Ber-poly}
\frac{x\,e^{tx}}{e^x-1} = \sum_{m=0}^\infty B_m(t)\frac{x^m}{m!}.
\end{equation}

The expansion
\[
\frac{e^s+1}{e^s-1} = \coth\frac{s}{2}
= \frac{2}{s} + \sum_{m\ge 1}\frac{2B_{2m}}{(2m)!} s^{2m-1}
\]
will be needed. We introduce the coefficients $r_{n,t}$.

\begin{dfn}
For nonnegative integers $n,t$, let $r_{n,t}$ be defined by the generating function
\begin{equation}\label{e-cot-expansion}
\left(\frac{e^s+1}{e^s-1} \right)^{n}
=\frac{2^{n}}{s^{n}}\Bigl(1+\sum_{m\geq 1} \frac{B_{2m} s^{2m}}{(2m)!}\Bigr)^{n}
=2^{n}\sum_{t\geq 0} r_{n,t}\,s^{2t-n}.
\end{equation}
An explicit formula for $r_{n,t}$ is
\[
r_{n,t} = \sum_{\substack{c_1+2c_2+\cdots+lc_l = t \\ c_i \in \mathbb{N}}}
\binom{n}{c_1,c_2,\dots,c_l,\,n-(c_1+\cdots+c_l)}
\prod_{m=1}^{l} \left( \frac{B_{2m}}{(2m)!} \right)^{c_m}.
\]
\end{dfn}
The coefficients $r_{n,t}$ can also be expressed in terms of higher-order Bernoulli polynomials;
Table~\ref{tab:r-coeff} lists the verified values for \(1\le n\le 5\) and $0\leq 2t \leq n$.

\begin{table}[htbp]
\centering
\caption{Verified values of the coefficients \(r_{n,t}\)}
\label{tab:r-coeff}
\renewcommand{\arraystretch}{1.8}
\begin{tabular}{c|ccccc}
\hline
\(t \setminus n\) & \(n=1\) & \(n=2\) & \(n=3\) & \(n=4\) & \(n=5\) \\
\hline
\(t=0\) & \(1\) & \(1\) & \(1\) & \(1\) & \(1\) \\[2pt]
\(t=1\) & --- & \(\dfrac{1}{6}\) & \(\dfrac{1}{4}\) & \(\dfrac{1}{3}\) & \(\dfrac{5}{12}\) \\[2pt]
\(t=2\) & --- & --- & --- & \(\dfrac{13}{360}\) & \(\dfrac{1}{16}\) \\
\hline
\end{tabular}
\end{table}

We also introduce the Euler polynomials $E_m(x)$, defined by
\begin{equation}\label{e-Euler-poly}
\frac{2 e^{xt}}{e^t + 1} = \sum_{m=0}^{\infty} \frac{E_m(x)}{m!} t^m.
\end{equation}

\subsection{Constant term notation}
\label{subsec:CT}

In this subsection we collect several constant term notations needed in our discussion.

A formal Laurent series over a ring $R$,
\[
L(\lambda) = \sum_{n\in \mathbb{Z}} c_n \lambda^n, \qquad c_n \in R,
\]
is called a Laurent series if there exists an integer $N$ such that $c_n=0$ for all $n<N$.
The constant term operator $\CT_\lambda$ extracts the coefficient $c_0$, while the residue operator $\Res_\lambda$ extracts $c_{-1}$:
\[
\CT_\lambda L(\lambda)=c_0, \qquad \Res_\lambda L(\lambda)= c_{-1}.
\]
Although the two operators can be transformed into each other, we find the constant term more natural.

A rational function $E(\lambda)$ may admit several series expansions, so when applying $\CT_\lambda$ one must specify its expansion.
For example, expanding $E(\lambda)$ at $\lambda=u$ as
\[
E(\lambda)= \sum_{n=N}^{\infty} c_n (\lambda-u)^n,
\]
its residue at $\lambda=u$ is $\Res_{\lambda=u} E(\lambda)=c_{-1}$. The analogous constant term is defined differently.

Consider the partial fraction decomposition (PFD) of a rational function $E(\lambda)$:
\[
E(\lambda) = \frac{1}{(1-u \lambda)^a} F(\lambda) = \frac{A(\lambda)}{(1-u \lambda)^a} + \text{other terms},
\]
where $u$ is independent of $\lambda$, $F(u^{-1})$ exists, and $A(\lambda)$ is a polynomial in $\lambda$ of degree less than $a$.
Then the constant term of $E(\lambda)$ at $u^{-1}$ is defined as
\[
\CT_{\lambda=u^{-1}} E(\lambda) = A(0).
\]
In particular, when $a=1$, $A(\lambda)$ is the constant $F(u^{-1})$.

This notation originated in the computation of the Ehrhart series of the Birkhoff polytope; its connection with residues
was found in \cite{XinZhang2024}.

\begin{lem}\label{lem:CT-res}
Let $E(\lambda)$ be as above. Then
\[
\CT_{\lambda=u^{-1}}\bigl(\lambda E(\lambda)\bigr)=-\Res_{\lambda=u^{-1}} E(\lambda).
\]
\end{lem}

Throughout the paper, $\zeta_a = e^{2\pi i / a}$ denotes the primitive $a$-th root of unity.

The next constant term notation does not have a residue counterpart in the literature.
Consider the variational PFD
\[
E(\lambda) = \frac{1}{1-u \lambda^a} F(\lambda) = \frac{A(\lambda)}{1-u \lambda^a} + \text{other terms},
\]
with the same conditions except that we require $F(u^{-1/a} \zeta_a^j)$ to exist for all $j$.
Then we define
\[
\CT_{\lambda} \frac{1}{\underline{1-u \lambda^a}} F(\lambda) =A(0).
\]

It is easy to see that
\[
A(0) = \sum_{j=0}^{a-1} \CT_{\lambda=u^{-1/a} \zeta_a^j} E(\lambda)
      = \sum_{j=0}^{a-1} \frac{1-u^{1/a} \zeta_a^{-j} \lambda}{1-u \lambda^a} F(\lambda)\Big|_{\lambda=u^{-1/a} \zeta_a^j}.
\]
Using L'H\^opital's rule gives the following basic result.

\begin{lem}\label{lem:underlined-CT}
Let $F(\lambda)$ be a rational function such that $F(u^{-1/a}\zeta_a^j)$ exists for all $j=0,1,\dots,a-1$. Then
\[
\frac{1}{a}\sum_{k=0}^{a-1} F(u^{-1/a}\zeta_a^k) = \CT_\lambda \underline{\frac{1}{1-u\lambda^a}} F(\lambda).
\]
\end{lem}
This identity and its combinatorial interpretations, including its connection with simplicial cones, are discussed in detail in \cite{XinX25}.


We conclude this subsection with the following special change-of-variable formula.

\begin{prop}\label{prop:change-var}
Suppose $E(\lambda)$ is a rational function with a pole at $\lambda=u$.
Then
\[
\CT_{\lambda=u} E(\lambda) = -\CT_{s} s\,E(u e^s).
\]
\end{prop}

We give two proofs. The first is for readers familiar with residue computations.
\begin{proof}[First proof]
We have
\begin{align*}
  \CT_{\lambda=u} E(\lambda) &= -\Res_{\lambda=u} \lambda^{-1} E(\lambda)\\
                             &= -\Res_{\lambda=0} (\lambda+u)^{-1} E(\lambda+u)
                             = -\Res_{\lambda} (\lambda+u)^{-1} E(\lambda+u),
\end{align*}
where $(\lambda+u)^{-1} E(\lambda+u)$ is expanded as a Laurent series in $\lambda$.  Next apply Jacobi's change of variable
$\lambda = u(e^s-1)$ (so that $d\lambda = u e^s ds$) to obtain
\begin{align*}
  \CT_{\lambda=u} E(\lambda) &= -\Res_{s} \frac{1}{u e^s} E(u e^s)\, u e^s
                             = -\CT_{s} s\,E(u e^s),
\end{align*}
as desired.
\end{proof}

The second proof is by direct constant term manipulation.
\begin{proof}[Second proof]
By linearity, it suffices to prove the identity for $E(\lambda)=(1-\lambda/u)^{-k}$ with integer $k$. That is,
we need to show
\[
\CT_s \frac{-s}{(1-e^s)^k} =
\begin{cases}
1, & k\ge 1,\\
0, & k\le 0.
\end{cases}
\]
The case $k\le 0$ is trivial; we prove the positive case by induction on $k$.

For $k=1$ the identity can be verified directly.  Assume it holds for $k=j$ and consider $k=j+1$.
Using $\CT_s\, s\,G'(s)=0$ for any formal Laurent series $G(s)$, we apply it to $G(s)=\frac{1}{j(1-e^s)^j}$:
\[
0=\CT_s s\left( \frac{1}{j(1-e^s)^{j}}\right){\!'} = \CT_s \frac{s e^s}{(1-e^s)^{j+1}}.
\]
Hence
\[
\CT_s \frac{-s}{(1-e^s)^{j+1}}
   = \CT_s \frac{-s}{(1-e^s)^{j+1}} + \CT_s \frac{s e^s}{(1-e^s)^{j+1}}
   = \CT_s \frac{-s}{(1-e^s)^{j}} = 1,
\]
by the induction hypothesis.
\end{proof}

\subsection{A constant term expression for weighted cotangent sums}
\label{subsec:Cot-CT}

Let $a,n,k\in\mathbb{N}$, $\theta\in\mathbb{R}$, and $0\le a\le k-1$.  Define
\begin{equation}\label{eq:cot-def}
 \Cot(a,n,k;\theta)=\sum_{\substack{j=0\\ j\neq j_0}}^{k-1} (\zeta_k^j)^a\cot^{n}\Bigl(\theta+\frac{\pi j}{k}\Bigr),
\end{equation}
where $j_0$ is the unique index (if it exists) for which $\theta+\frac{\pi j_0}{k}\equiv 0 \pmod \pi$.
For example, when $\theta=0$ we have $j_0=0$, and
\[
\Cot(a,n,k;0)=:\Cot(a,n,k)=\sum_{j=1}^{k-1} (\zeta_k^j)^a\cot^{n}\Bigl(\frac{\pi j}{k}\Bigr).
\]
This subsection establishes the following constant term representation.

\begin{prop}\label{prop:Cot-CT}
Let $k\geq 2$ be an integer and suppose $\dfrac{k\theta}{\pi}\notin\mathbb{Z}$, then there is no singularity and
\[
\Cot(a,n,k;\theta)
= (-i)^n\Bigl(
k\delta_{a,0}
+ \CT_s\Bigl(
s \cdot \frac{k e^{a(s-2i\theta)}}{1-e^{k(s-2i\theta)}}
\cdot \frac{(1+e^{s})^n}{(1-e^{s})^n}
\Bigr)\Bigr),
\]
where $\delta_{a,b}$ is the Kronecker delta.  Otherwise, a singularity exists with $j_0 \equiv -\frac{k\theta}{\pi} \pmod{k}$, and
\[
\begin{aligned}
\Cot(a,n,k;\theta)
&= \zeta_{k}^{a j_0}\, \Cot(a,n,k) \\
&= \zeta_k^{a j_0} (-i)^n \Bigl(
k\delta_{a,0}
- \CT_u \CT_s\, s\cdot \frac{k e^{a(s-u)}}{1-e^{k(s-u)}} \cdot \frac{(1+e^s)^n}{(1-e^s)^n}
- \CT_u \frac{(1+e^u)^n}{(1-e^u)^n}
\Bigr).
\end{aligned}
\]
\end{prop}

\begin{proof}
Using Euler's formula $e^{ix}=\cos x+i\sin x$ we obtain
\[
\cot x = \frac{\cos x}{\sin x} = i\,\frac{e^{ix}+e^{-ix}}{e^{ix}-e^{-ix}} = -i\,\frac{1+e^{2ix}}{1-e^{2ix}}.
\]
Substituting $x=\theta +\pi j/k$ and setting $z=e^{2i\theta}$ gives
\[
\cot\Bigl(\theta+\frac{\pi j}{k}\Bigr) = -i\,\frac{1+z\zeta_k^{\,j}}{1-z\zeta_k^{\,j}}.
\]
Raising to the $n$-th power, multiplying by $\zeta_k^{aj}$, and summing over $j$ (excluding $j_0$ if it exists) yields
\[
\sum_{\substack{j=0\\ j\neq j_0}}^{k-1} \zeta_k^{aj} \cot^{n}\Bigl(\theta+\frac{\pi j}{k}\Bigr)
= (-i)^n \sum_{\substack{j=0\\ j\neq j_0}}^{k-1} \zeta_k^{aj} \Bigl(\frac{1+z\zeta_k^{\,j}}{1-z\zeta_k^{\,j}}\Bigr)^{\!n}
= (-i)^n \sum_{\substack{j=0\\ j\neq j_0}}^{k-1} F(\zeta_k^{\,j}),
\]
where
\[
F(\lambda) = \frac{\lambda^a (1+z\lambda)^n}{(1-z\lambda)^n}.
\]

We distinguish two cases.

\noindent\textbf{Case 1: $j_0$ does not exist.}
Applying Lemma~\ref{lem:underlined-CT} gives
\[
\sum_{j=0}^{k-1} F(\zeta_k^{\,j})
= k\,\CT_\lambda \underline{\frac{1}{1-\lambda^k}} F(\lambda)
= \CT_\lambda \underline{\frac{k\lambda^a}{1-\lambda^k}} \frac{(1+z\lambda)^n}{(1-z\lambda)^n}.
\]
Consider the PFD of the proper rational function (using $0\le a\le k-1$)
\[
f(\lambda) = \frac{k}{1-\lambda^k}F(\lambda) = \frac{k\lambda^{a}(1+\lambda z)^n}{(1-\lambda^k)(1-\lambda z)^n}
= \frac{A(\lambda)}{1-\lambda^k} + \frac{B(\lambda)}{(1-z\lambda)^n}.
\]

We need $A(0)$, which we obtain via $B(0)$:
\[
A(0) = f(0) - B(0) = k\delta_{a,0} - \CT_{\lambda=z^{-1}} f(\lambda).
\]
By Proposition~\ref{prop:change-var},
\begin{align*}
B(0) &= \CT_{\lambda=z^{-1}} \frac{1}{(1-\lambda z)^n}\frac{k\lambda^a(1+\lambda z)^n}{1-\lambda^k} \\
     &= \CT_s \frac{-s}{(1-e^s)^n}\frac{k z^{-a} e^{as}(1+e^s)^n}{1-z^{-k}e^{ks}}.
\end{align*}
Consequently,
\begin{equation}\label{eq:sum-F}
\sum_{j=0}^{k-1} F(\zeta_k^{\,j}) = k\delta_{a,0} + \CT_s \frac{s}{(1-e^s)^n}\frac{k z^{-a} e^{as}(1+e^s)^n}{1-z^{-k}e^{ks}}.
\end{equation}
Replacing $z$ by $e^{2i\theta}$ and multiplying by $(-i)^n$ gives the first identity.

\noindent\textbf{Case 2: $j_0$ exists.}
This means $z\zeta_k^{j_0}=1$.  First reduce to $\theta=0$:
\[
F(\zeta_k^j) = \frac{\zeta_k^{aj}(1+z\zeta_k^j)^n}{(1-z\zeta_k^j)^n}
= \frac{\zeta_k^{aj_0}\,\zeta_k^{a(j-j_0)}(1+\zeta_k^{j-j_0})^n}{(1-\zeta_k^{j-j_0})^n}.
\]
Since $\zeta_k^k=1$,
\begin{align*}
\Cot(a,n,k;\theta)
&= (-i)^n \zeta_k^{aj_0} \sum_{\substack{j=0\\ j\neq j_0}}^{k-1}
   \frac{\zeta_k^{a(j-j_0)}(1+\zeta_k^{j-j_0})^n}{(1-\zeta_k^{j-j_0})^n} \\
&= (-i)^n \zeta_k^{aj_0} \sum_{j=1}^{k-1}
   \frac{\zeta_k^{aj}(1+\zeta_k^j)^n}{(1-\zeta_k^j)^n}
 = \zeta_k^{aj_0} \Cot(a,n,k).
\end{align*}

To evaluate $\Cot(a,n,k)$, introduce a slack variable $\tilde{z}\to 1$ playing the role of $z$ in Case~1. Then
\begin{align*}
\Cot(a,n,k)
&= \lim_{\tilde{z}\to 1} (-i)^n \sum_{j=1}^{k-1} F(\zeta_k^j)\big|_{z=\tilde{z}} \\
&= (-i)^n \lim_{\tilde{z}\to 1} \Bigl(\sum_{j=0}^{k-1} F(\zeta_k^j) - F(1)\Bigr)\Big|_{z=\tilde{z}} \\
&\stackrel{\eqref{eq:sum-F}}{=} (-i)^n \lim_{\tilde{z}\to 1}
   \Bigl( k\delta_{a,0}
        - \CT_s \frac{s}{(1-e^s)^n}\frac{k \tilde{z}^{-a} e^{as}(1+e^s)^n}{1-\tilde{z}^{-k}e^{ks}}
        - \frac{(1+\tilde{z})^n}{(1-\tilde{z})^n} \Bigr) \\
&\stackrel{\tilde{z}=e^u}{=} (-i)^n \lim_{u\to 0}
   \Bigl( k\delta_{a,0}
        - \CT_s \frac{s}{(1-e^s)^n}\frac{k e^{-au} e^{as}(1+e^s)^n}{1-e^{-ku}e^{ks}}
        - \frac{(1+e^u)^n}{(1-e^u)^n} \Bigr) \\
&= (-i)^n \Bigl( k\delta_{a,0}
             - \CT_u \CT_s \frac{s}{(1-e^s)^n}\frac{k e^{a(s-u)}(1+e^s)^n}{1-e^{k(s-u)}}
             - \CT_u \frac{(1+e^u)^n}{(1-e^u)^n} \Bigr).
\end{align*}
In the last step the constant term with respect to $u$ is extracted.
\end{proof}

\section{Cotangent Power Sums}
\label{sec:cot}

\subsection{Explicit formulas for even and odd powers}
\label{subsec:formulas}

\begin{thm}\label{thm:cot-even}
Let $k\ge 2$, $n\ge 1$ be integers, and $a\in\{0,1,\dots,k-1\}$. Then
\[
\Cot(a,2n,k)=\sum_{j=1}^{k-1} (\zeta_k^j)^a\cot^{2n}\Bigl(\frac{\pi j}{k}\Bigr)
=(-1)^n\Bigl(
k\delta_{a,0}-2^{2n}\sum_{t=0}^{n} \frac{B_{2(n-t)}\bigl(\frac{a}{k}\bigr)}{(2(n-t))!}\, r_{2n,t}\, k^{2(n-t)}
\Bigr).
\]
\end{thm}

\begin{proof}
Set $\theta=0$ and replace $n$ by $2n$ in~\eqref{eq:cot-def}. The singular index is $j_0=0$, so
\[
\Cot(a,2n,k;0) = \Cot(a,2n,k) = \sum_{j=1}^{k-1} \bigl(\zeta_k^{j}\bigr)^{a} \cot^{2n}\Bigl(\frac{\pi j}{k}\Bigr).
\]
By Proposition~\ref{prop:Cot-CT},
\begin{equation}\label{eq:cot-even-intermediate}
\Cot(a,2n,k) = (-1)^n \Bigl(
k\delta_{a,0}
- \CT_u \CT_s \Bigl( s \cdot \frac{k e^{a(s-u)}}{1 - e^{k(s-u)}} \cdot \frac{(1+e^s)^{2n}}{(1-e^s)^{2n}} \Bigr)
- \CT_u \Bigl( \frac{(1+e^u)^{2n}}{(1-e^u)^{2n}} \Bigr)
\Bigr).
\end{equation}

We first evaluate the iterated constant term: expand the operand in $s$, take $\CT_s$, then expand in $u$ and take $\CT_u$.
The $u$-free part is expanded via~\eqref{e-cot-expansion} (with $n\mapsto 2n$):
\begin{equation}\label{eq:u-free}
\Bigl(\frac{1+e^s}{1-e^s}\Bigr)^{2n}
= 2^{2n} \sum_{t \ge 0} r_{2n,t}\, s^{2t-2n}.
\end{equation}
The $u$-dependent part uses~\eqref{e-Ber-poly} with $x=k(s-u)$ and $t=a/k$:
\begin{equation}\label{eq:u-dep}
\frac{k e^{a(s-u)}}{1 - e^{k(s-u)}}
= -\sum_{m=0}^{\infty} \frac{k^m B_m\!\bigl(\frac{a}{k}\bigr)}{m!} (s-u)^{m-1}
= -\frac{1}{s-u} - \sum_{m=1}^{\infty} \frac{k^m B_m\!\bigl(\frac{a}{k}\bigr)}{m!} (s-u)^{m-1}.
\end{equation}

Each summand in the second term is a polynomial in $s$ and $u$; the first term expands as
\[
\frac{1}{s-u} = -\frac{1}{u}\cdot\frac{1}{1-s/u} = -\sum_{j\ge 0}\frac{s^{j}}{u^{\,j+1}},
\]
which is valid when $|s/u|<1$. This is the case since we treat $u$ as constant when applying $\CT_s$.

Taking $\CT_u$ first (justified on the series level\footnote{An iterated constant term is properly interpreted in the field $\mathbb{C}((u))((s))$ of iterated Laurent series; see \cite{Xin2004}.}) gives
\[
\CT_u \Bigl( -\frac{1}{s-u} - \sum_{m=1}^{\infty} \frac{k^m B_m\!\bigl(\frac{a}{k}\bigr)}{m!} (s-u)^{m-1} \Bigr)
= -\sum_{m\ge 1}\frac{k^m B_m\!\bigl(\frac{a}{k}\bigr) s^{m-1}}{m!}.
\]
Combining with~\eqref{eq:u-free}, the iterated constant term becomes
\begin{align*}
\CT_s\biggl( s \cdot \frac{k e^{a(s-u)}}{1-e^{k(s-u)}} \cdot \frac{(1+e^s)^{2n}}{(1-e^s)^{2n}} \biggr)
&= \CT_s\biggl( -\frac{2^{2n}}{s^{2n-1}} \sum_{t\ge0} r_{2n,t} s^{2t}
                 \sum_{m\ge0} \frac{k^{m+1} B_{m+1}\!\bigl(\frac{a}{k}\bigr)}{(m+1)!} s^{m} \biggr)\\
&= -2^{2n} \sum_{t=0}^{n-1} \frac{B_{2(n-t)}\!\bigl(\frac{a}{k}\bigr)}{(2(n-t))!}\, r_{2n,t}\, k^{2(n-t)},
\end{align*}
where we extracted the coefficient of $s^{2n-1}$ from $2t+m=2n-1$, i.e., $m=2(n-t)-1$.

The remaining term $\frac{(1+e^u)^{2n}}{(1-e^u)^{2n}}$ in~\eqref{eq:cot-even-intermediate} is handled by substituting $s\mapsto u$ in~\eqref{e-cot-expansion}:
\[
\Bigl(\frac{1+e^u}{1-e^u}\Bigr)^{2n}
= 2^{2n} \sum_{t\ge 0} r_{2n,t}\, u^{2t-2n},
\]
whose constant term is $2^{2n} r_{2n,n}$ (attained at $t=n$).
Collecting all constant terms and multiplying by $(-1)^n$ completes the proof.
\end{proof}

\begin{cor}
Let $k\ge 2$, $n\ge 1$, and $a\in\{0,1,\dots,k-1\}$. Then
\[
\sum_{j=1}^{k-1} \cos\Bigl(\frac{2\pi a j}{k}\Bigr) \cot^{2n}\Bigl(\frac{\pi j}{k}\Bigr)
= \Re\bigl(\Cot(a,2n,k)\bigr) = \Cot(a,2n,k),
\]
and
\[
\sum_{j=1}^{k-1} \sin\Bigl(\frac{2\pi a j}{k}\Bigr) \cot^{2n}\Bigl(\frac{\pi j}{k}\Bigr)
= \Im\bigl(\Cot(a,2n,k)\bigr) = 0.
\]
\end{cor}
\begin{proof}
By definition,
\[
\Cot(a,2n,k) = \sum_{j=1}^{k-1} e^{2\pi i a j /k} \cot^{2n}\Bigl(\frac{\pi j}{k}\Bigr).
\]
The real part recovers the cosine-weighted sum; the imaginary part vanishes because $\Cot(a,2n,k)$ is real.
\end{proof}

\begin{thm}\label{thm:cot-odd}
Let $k\ge 2$, $n\in\mathbb{N}$, and $a\in\{0,1,\dots,k-1\}$. Then
\[
\Cot(a,2n+1,k)=\sum_{j=1}^{k-1} (\zeta_k^j)^a\cot^{2n+1}\Bigl(\frac{\pi j}{k}\Bigr)
=(-1)^{n+1}\, i\Bigl(
k\delta_{a,0}+2^{2n+1}\sum_{t=0}^{n} \frac{B_{2(n-t)+1}\bigl(\frac{a}{k}\bigr)}{(2(n-t)+1)!}\, r_{2n+1,t}\, k^{2(n-t)+1}
\Bigr).
\]
\end{thm}
\begin{proof}
Set $\theta=0$ and replace $n$ by $2n+1$ in~\eqref{eq:cot-def}. The singular index is $j_0=0$, so
\[
\Cot(a,2n+1,k;0) = \Cot(a,2n+1,k) = \sum_{j=1}^{k-1} \bigl(\zeta_k^j\bigr)^a \cot^{2n+1}\Bigl(\frac{\pi j}{k}\Bigr).
\]
The rest of the proof follows the same steps as in Theorem~\ref{thm:cot-even}; we omit the details.
\end{proof}

\begin{cor}
Let $0\le a<k$. Then
\[
\sum_{j=1}^{k-1}\cos\!\Bigl(\frac{2\pi a j}{k}\Bigr)\cot^{2n+1}\Bigl(\frac{\pi j}{k}\Bigr)=0,
\]
and
\[
\sum_{j=1}^{k-1}\sin\!\Bigl(\frac{2\pi a j}{k}\Bigr)\cot^{2n+1}\Bigl(\frac{\pi j}{k}\Bigr)
=(-1)^{n+1}\Bigl(
k\delta_{a,0}+2^{2n+1}\sum_{t=0}^{n} \frac{B_{2(n-t)+1}\bigl(\frac{a}{k}\bigr)}{(2(n-t)+1)!}\, r_{2n+1,t}\, k^{2(n-t)+1}
\Bigr).
\]
\end{cor}
\begin{proof}
Since $\Cot(a,2n+1,k)$ is purely imaginary, taking real and imaginary parts yields the stated identities.
\end{proof}

\section{Generalizations and Applications}
\label{sec:apps}

\subsection{Tangent power sums with root-of-unity weights}
\label{subsec:tan}

Given the formulas for cotangent power sums, it is natural to seek analogous results for tangent powers.
For a real parameter $\theta$, define
\[
\Tan(a,n,k;\theta):=\sum_{\substack{j=0\\ j\neq j_0}}^{k-1} (\zeta_k^j)^a \tan^n\Bigl(\theta+\frac{\pi j}{k}\Bigr),
\]
where $j_0$ is the index (if it exists) for which $\theta+\frac{\pi j_0}{k}\equiv\frac{\pi}{2}\pmod{\pi}$.
Since $\tan x = -\cot(x+\frac{\pi}{2})$,
\begin{equation}\label{eq:cot-tan}
\Tan(a,n,k;\theta)=(-1)^n\Cot\Bigl(a,n,k;\theta+\frac{\pi}{2}\Bigr).
\end{equation}

One might attempt to mimic $\Cot(a,n,k)$ simply by replacing $\cot$ with $\tan$ in the sum without the $\theta$ parameter. However, when $k$ is even, the term with $j=k/2$ blows up, so a direct analogue is not possible. It is therefore more natural to start from the $\theta$-dependent definition above and then specialize. Accordingly, we set
\[
\Tan(a,n,k):=\Tan(a,n,k;0)=\sum_{\substack{j=0\\ j\neq j_0}}^{k-1} (\zeta_k^j)^a \tan^n\Bigl(\frac{\pi j}{k}\Bigr),
\]
where $j_0$ exists precisely when $k$ is even (in which case $j_0=k/2$). Note that $j=0$ contributes nothing because $\tan 0=0$.

\begin{thm}\label{thm:tan}
Let $k\ge2$, $n\ge1$ be integers, $a\in\{0,1,\dots,k-1\}$, and let $E_m(x)$ be the $m$-th Euler polynomial defined by~\eqref{e-Euler-poly}. Then:
\begin{enumerate}
    \item If $k$ is even,
    \begin{equation}\label{eq:Tan-even}
    \Tan(a,n,k)=(-1)^{a+n}\Cot(a,n,k),
    \end{equation}
    whose explicit value is given by Theorem~\ref{thm:cot-even} (or \ref{thm:cot-odd} depending on parity).
    \item If $k$ is odd,
    \begin{equation}\label{eq:Tan-odd}
    \Tan(a,n,k)=i^{\,n}\Bigl( k\delta_{a,0}+(-1)^{a+n}\,2^{n-1}\sum_{t=0}^{\lfloor (n-1)/2\rfloor} \frac{E_{n-2t-1}\bigl(\frac{a}{k}\bigr)}{(n-2t-1)!}\, r_{n,t}\, k^{n-2t}\Bigr).
    \end{equation}
\end{enumerate}
\end{thm}

\begin{proof}
At $\theta=0$, \eqref{eq:cot-tan} gives
\begin{align}
\Tan(a,n,k)=(-1)^n\Cot\Bigl(a,n,k;\frac{\pi}{2}\Bigr). \label{e-4.2}
\end{align}
We separate the two cases according to Proposition~\ref{prop:Cot-CT} with $\theta=\frac{\pi}{2}$.

\noindent\textbf{Case 1: $k$ even.}
The shifted cotangent sum has a simple pole at $j_0=k/2$ because
$\theta+\frac{\pi j_0}{k}=\pi\equiv0\pmod{\pi}$.
Applying the singular case of Proposition~\ref{prop:Cot-CT} together with $\zeta_k^{j_0}=-1$,
\[
\Cot\Bigl(a,n,k;\frac{\pi}{2}\Bigr)=\zeta_k^{a j_0}\Cot(a,n,k)=(-1)^a\Cot(a,n,k).
\]
Combined with \eqref{e-4.2} we obtain $\Tan(a,n,k)=(-1)^{a+n}\Cot(a,n,k)$, i.e., \eqref{eq:Tan-even}.

\noindent\textbf{Case 2: $k$ odd.}
Because $k$ is odd, $-\frac{k\theta}{\pi}=-\frac{k}{2}$ is not an integer, so no singularity occurs.
Using the non-singular case of Proposition~\ref{prop:Cot-CT} with $\theta=\frac{\pi}{2}$,
\begin{equation}\label{eq:cot-halfpi-odd}
\Cot\Bigl(a,n,k;\frac{\pi}{2}\Bigr)=(-i)^n\Bigl( k\delta_{a,0}+(-1)^a\CT_s\Bigl( s\cdot\frac{ke^{as}}{1+e^{ks}}\cdot\frac{(1+e^s)^n}{(1-e^s)^n}\Bigr)\Bigr).
\end{equation}
Now expand the two factors. From~\eqref{e-Euler-poly} with $t=ks$ and $x=\frac{a}{k}$,
\[
\frac{k e^{as}}{1+e^{ks}}=\frac12\sum_{m=0}^{\infty}\frac{E_m\bigl(\frac{a}{k}\bigr)}{m!}\,k^{m+1}s^{m}.
\]

For the cotangent-type factor, \eqref{e-cot-expansion} gives
\[
\Bigl(\frac{1+e^s}{1-e^s}\Bigr)^{\!n}=(-2)^n\sum_{t\ge0} r_{n,t}\,s^{2t-n}.
\]
Multiplying by $s$ and extracting the constant term in $s$ requires
\[
1+m+(2t-n)=0 \;\Longleftrightarrow\; m=n-2t-1,
\]
so only $0\le t\le\lfloor (n-1)/2\rfloor$ contribute. Hence
\[
\CT_s\Bigl( s\cdot\frac{ke^{as}}{1+e^{ks}}\cdot\frac{(1+e^s)^n}{(1-e^s)^n}\Bigr)
=(-1)^n\,2^{n-1}\sum_{t=0}^{\lfloor (n-1)/2\rfloor} \frac{E_{n-2t-1}\bigl(\frac{a}{k}\bigr)}{(n-2t-1)!}\, r_{n,t}\, k^{n-2t}.
\]
Substituting this into \eqref{eq:cot-halfpi-odd} and using \eqref{e-4.2},
\[
\Tan(a,n,k)=(-1)^n\cdot(-i)^n\Bigl( k\delta_{a,0}+(-1)^a\cdot(-1)^n\,2^{n-1}\sum_{t=0}^{\lfloor (n-1)/2\rfloor} \frac{E_{n-2t-1}\bigl(\frac{a}{k}\bigr)}{(n-2t-1)!}\, r_{n,t}\, k^{n-2t}\Bigr).
\]
Since $(-1)^n(-i)^n=i^{\,n}$, this simplifies to \eqref{eq:Tan-odd}.
\end{proof}

\subsection{Sums of even powers of cosecant and secant}
\label{subsec:cscsec}

We introduce twisted Dowker-type sums (\cite{Dowker1987,Cvijovic2012}):
\[
\Csc(a,2n,k;\theta)=\sum_{\substack{j=0\\ j\neq j_0}}^{k-1} \zeta_k^{aj}\, \csc^{2n}\Bigl(\theta+\frac{\pi j}{k}\Bigr),\qquad
\Sec(a,2n,k;\theta)=\sum_{\substack{j=0\\ j\neq j_0}}^{k-1} \zeta_k^{aj}\, \sec^{2n}\Bigl(\theta+\frac{\pi j}{k}\Bigr),
\]
where $j_0$ is the singular index for the respective trigonometric function.

\begin{thm}\label{thm:csc}
For $k\ge2$, $n\ge1$, and $a\in\{0,1,\dots,k-1\}$,
\[
\Csc(a,2n,k;\theta)
= \sum_{\substack{j=0\\ j\neq j_0}}^{k-1} \zeta_k^{aj} \csc^{2n}\Bigl(\theta+\frac{\pi j}{k}\Bigr)
= \sum_{s=0}^{n} \binom{n}{s}\, \Cot(a, 2s, k;\theta).
\]
\end{thm}

\begin{proof}
From $\csc^2 x = 1 + \cot^2 x$ and the binomial theorem,
\[
\csc^{2n} x = (1+\cot^2 x)^n = \sum_{s=0}^n \binom{n}{s} \cot^{2s} x.
\]
Multiplying by $\zeta_k^{aj}$ and summing over admissible $j$ gives
\[
\sum_{\substack{j=0\\ j\neq j_0}}^{k-1} \zeta_k^{aj} \csc^{2n}\Bigl(\theta+\frac{\pi j}{k}\Bigr)
= \sum_{s=0}^n \binom{n}{s} \sum_{\substack{j=0\\ j\neq j_0}}^{k-1} \zeta_k^{aj} \cot^{2s}\Bigl(\theta+\frac{\pi j}{k}\Bigr).
\]
The inner sum is exactly $\Cot(a,2s,k;\theta)$. This completes the proof.
\end{proof}

\begin{thm}\label{thm:sec}
Let $k\ge 2$, $n\ge1$, and $a\in\{0,1,\dots,k-1\}$. Then
\[
\Sec(a,2n,k;\theta)
= \sum_{\substack{j=0\\ j\neq j_0}}^{k-1} \zeta_k^{aj} \sec^{2n}\Bigl(\theta+\frac{\pi j}{k}\Bigr)
= \sum_{s=0}^{n} \binom{n}{s}\, \Tan(a,2s,k;\theta).
\]
\end{thm}
\begin{proof}
Analogous to the cosecant case, using $\sec^2 x = 1 + \tan^2 x$.
\end{proof}

\subsection{Examples and special cases}
\label{subsec:examples}

The corollaries gathered here are all direct consequences of the general theorems
in Sections~\ref{sec:cot} and~\ref{sec:apps}; they are classical identities, not new results.
We reproduce them in a unified way by specialising parameters in Theorems~\ref{thm:cot-even}--\ref{thm:tan} and using the
binomial expansions for $\csc^{2n}x$ and $\sec^{2n}x$.

\begin{cor}
\label{cor:ordinary}
For integers $k\ge2$, $n\in\mathbb{N}_+$,
\[
\sum_{j=1}^{k-1} \cot^{2n}\Bigl(\frac{\pi j}{k}\Bigr)
= (-1)^n\Bigl(
 k-2^{2n}\sum_{t=0}^{n} \frac{B_{2(n-t)}}{(2(n-t))!}\, r_{n,t}\, k^{2(n-t)}
\Bigr).
\]
\end{cor}
\begin{proof}
Set $a=0$, $\theta=0$ in Theorem~\ref{thm:cot-even}. The singular index is $j_0=0$, and $B_m(0)=B_m$.
\end{proof}

\begin{cor}
\label{cor:square}
For $k\ge2$,
\[
\sum_{j=1}^{k-1} \cot^{2}\Bigl(\frac{\pi j}{k}\Bigr)
= \Cot(0,2,k)=\frac{(k-1)(k-2)}{3}.
\]
\end{cor}
\begin{proof}
Take $n=1$ in Corollary~\ref{cor:ordinary}, with $r_{1,0}=1$ and $B_2=1/6$.
\end{proof}

\begin{cor}
\label{cor:csc2}
Let $k\ge2$, $a\in\{0,1,\dots,k-1\}$, $\zeta_k=e^{2\pi i/k}$. Then
\[
\sum_{j=1}^{k-1}\zeta_k^{aj}\csc^2\Bigl(\frac{\pi j}{k}\Bigr)
=2 a(a-k)+\frac{k^2}{3}-\frac{1}{3}.
\]
\end{cor}
\begin{proof}
With $n=1$, $\theta=0$, $j_0=0$ in Theorem~\ref{thm:csc}, the sum equals $\Cot(a,0,k)+\Cot(a,2,k)$.
Here $\Cot(a,0,k)=\sum_{j=1}^{k-1}\zeta_k^{aj}$ equals $k-1$ when $a=0$ and $-1$ when $1\le a\le k-1$.
Using Theorem~\ref{thm:cot-even} with $n=1$, $B_2(x)=x^2-x+\frac16$, $r_{2,0}=1$, $r_{2,1}=1/6$, we obtain
\[
\Cot(a,2,k)=2a(a-k)+\frac{k^2}{3}+\frac{2}{3}\qquad(0\le a\le k-1).
\]
Adding the two contributions yields the stated formula, which holds for all $a$ in the range.
\end{proof}

\begin{cor}
It holds that
\[
\sum_{j=0}^{k-1} \csc^2\Bigl(\frac{\pi(2j+1)}{2k}\Bigr) = k^2.
\]
\end{cor}
\begin{proof}
Take $a=0$ in Corollary~\ref{cor:csc2} and use the identity
\[
\sum_{j=0}^{k-1}\csc^2\Bigl(\frac{\pi(2j+1)}{2k}\Bigr)
= \sum_{j=1}^{2k-1}\csc^2\Bigl(\frac{\pi j}{2k}\Bigr)-\sum_{j=1}^{k-1}\csc^2\Bigl(\frac{\pi j}{k}\Bigr)
= \frac{(2k)^2-1}{3}-\frac{k^2-1}{3}=k^2.
\]
\end{proof}

\begin{cor}
Let $k$ be an odd positive integer. Then
\[
\sum_{j=1}^{k-1} \tan^{2}\Bigl(\frac{\pi j}{k}\Bigr) = k(k-1).
\]
\end{cor}
\begin{proof}
Substitute $n=2$, $\theta=0$, $a=0$ into Theorem~\ref{thm:tan} (odd $k$). The formula simplifies to $k^2-k$.
\end{proof}

\begin{cor}
\label{cor:sec2}
Let $k\ge2$ be an integer and $a\in\{0,1,\dots,k-1\}$. Then
\[
\sum_{\substack{j=0\\ j\neq j_0}}^{k-1}\zeta_k^{aj}\sec^2\Bigl(\frac{\pi j}{k}\Bigr)=
\begin{cases}
(-1)^a\bigl(k^2-2ak\bigr), & k\ \text{odd},\\[6pt]
(-1)^a\Bigl(2a(a-k)+\dfrac{k^2}{3}-\dfrac{1}{3}\Bigr), & k\ \text{even}.
\end{cases}
\]
\end{cor}

\begin{proof}
Using $\sec^2 x = 1+\tan^2 x$ and the definition of $\Sec(a,2,k)$,
\[
\Sec(a,2,k)=\sum_{\substack{j=0\\ j\neq j_0}}^{k-1}\zeta_k^{aj}+\Tan(a,2,k).
\]

\noindent\textbf{Odd $k$:} No singularity. When $a=0$, the geometric sum is $k$ and $\Tan(0,2,k)=k^2-k$, giving $k^2$. When $1\le a\le k-1$, the geometric sum vanishes and $\Tan(a,2,k)=(-1)^a(k^2-2ak)$ (from Theorem~\ref{thm:tan}, case~2). Both cases are covered by $(-1)^a(k^2-2ak)$.

\noindent\textbf{Even $k$:} The singularity is $j_0=k/2$. We consider two cases. For $a\neq0$,
\[
\sum_{\substack{j=0\\ j\neq k/2}}^{k-1}\zeta_k^{aj}= -\zeta_k^{ak/2}=(-1)^{a+1}.
\]
By Theorem~\ref{thm:tan} (case~1), $\Tan(a,2,k)=(-1)^a\Cot(a,2,k)$. Using the value of $\Cot(a,2,k)$ from the proof of Corollary~\ref{cor:csc2},
\[
\Tan(a,2,k)=(-1)^a\Bigl(2a(a-k)+\frac{k^2}{3}+\frac{2}{3}\Bigr).
\]
Hence
\[
\Sec(a,2,k)=(-1)^{a+1}+(-1)^a\Bigl(2a(a-k)+\frac{k^2}{3}+\frac{2}{3}\Bigr)
=(-1)^a\Bigl(2a(a-k)+\frac{k^2-1}{3}\Bigr).
\]

For $a=0$, the sum reduces to
\[
\sum_{\substack{j=0\\ j\neq k/2}}^{k-1}\sec^2\Bigl(\frac{\pi j}{k}\Bigr)
=k-1+\Tan(0,2,k)=k-1+\Cot(0,2,k)=\frac{k^2-1}{3},
\]
by Corollary~\ref{cor:square}.
\end{proof}


\begin{cor}
\label{cor:quartic}
For $k\ge2$,
\[
\sum_{j=1}^{k-1} \cot^{4}\Bigl(\frac{\pi j}{k}\Bigr)
= \frac{(k-1)(k-2)(k^2+3k-13)}{45}.
\]
\end{cor}
\begin{proof}
Take $n=2$, $a=0$ in Theorem~\ref{thm:cot-even}, with $r_{4,0}=1$, $r_{4,1}=1/3$, $r_{4,2}=13/360$, $B_2=1/6$, $B_4=-1/30$.
\end{proof}

\begin{cor}
We have
\[
\sum_{\substack{j=1\\ j\ \text{odd}}}^{k-1}
\cot^{2}\Bigl(\frac{\pi j}{2k}\Bigr)\csc^{2}\Bigl(\frac{\pi j}{2k}\Bigr)
= \frac{k^2(k^2-1)}{6}.
\]
\end{cor}
\begin{proof}
Let $S$ denote the sum. Using $\csc^2 x = 1+\cot^2 x$,
\[
S = \sum_{\substack{j=1\\ j\ \text{odd}}}^{k-1}\cot^{2}\Bigl(\frac{\pi j}{2k}\Bigr)
    +\sum_{\substack{j=1\\ j\ \text{odd}}}^{k-1}\cot^{4}\Bigl(\frac{\pi j}{2k}\Bigr).
\]
Extending the sums to $j=1,\dots,2k-1$ and subtracting the even-index terms yields
\[
\sum_{j=1}^{2k-1}\cot^{2}\Bigl(\frac{\pi j}{2k}\Bigr)+\sum_{j=1}^{2k-1}\cot^{4}\Bigl(\frac{\pi j}{2k}\Bigr)
-\sum_{j=1}^{k-1}\cot^{2}\Bigl(\frac{\pi j}{k}\Bigr)-\sum_{j=1}^{k-1}\cot^{4}\Bigl(\frac{\pi j}{k}\Bigr)
= \frac{k^4-k^2}{3},
\]
by Corollaries~\ref{cor:square} and \ref{cor:quartic}. The left-hand side equals $2S$, hence $S=k^2(k^2-1)/6$.
\end{proof}

\begin{cor}\label{cot-ae}
Let $n\ge1$ and $k\ge2$ be even. Then
\[
\sum_{j=1}^{k-1} (-1)^j \cot^{2n}\Bigl(\frac{\pi j}{k}\Bigr)
= (-1)^{n+1} 2^{2n} \sum_{t=0}^{n} r_{2n,t}\, \frac{B_{2(n-t)}\bigl(\frac12\bigr)}{(2(n-t))!}\, k^{2(n-t)}.
\]
\end{cor}
\begin{proof}
Set $a=k/2$ in Theorem~\ref{thm:cot-even}.
\end{proof}

\begin{cor}
For even $k\ge2$,
\[
\sum_{j=1}^{k-1} (-1)^j \cot^{2}\Bigl(\frac{\pi j}{k}\Bigr) = -\frac{1}{6}k^2+\frac{2}{3}.
\]
\end{cor}
\begin{proof}
Take $n=1$ in the previous corollary, with $r_{2,0}=1$, $r_{2,1}=1/6$.
\end{proof}

\begin{cor}
For even $k\ge2$,
\[
\sum_{j=1}^{k-1} (-1)^j \cot^{4}\Bigl(\frac{\pi j}{k}\Bigr) = -\frac{7}{360}k^4+\frac{2}{9}k-\frac{26}{45}.
\]
\end{cor}
\begin{proof}
Take $n=2$ with $r_{4,1}=1/3$, $r_{4,2}=13/360$ in Corollary~\ref{cot-ae}.
\end{proof}

We now apply our formulas to a classical problem posed by Acton~\cite{AMM4632}. The original solution uses contour integration; our unified method gives a more elementary proof.

\begin{cor}
For any positive integer $k$,
\[
\sum_{m=0}^{k-1} (-1)^m \tan\Bigl(\frac{\pi(2m+1)}{4k}\Bigr) = (-1)^{k+1} k.
\]
\end{cor}

\begin{proof}
Theorem~\ref{thm:tan} (case~1) for even modulus states $\Tan(a,n,k)=(-1)^{a+n}\Cot(a,n,k)$. Replace $k$ by $4k$ and set $n=1$, $a=k$. The singular index is $j_0=2k$, giving
\[
\Tan(k,1,4k)=\sum_{j=1,\,j\neq2k}^{4k-1}\zeta_{4k}^{kj}\tan\Bigl(\frac{\pi j}{4k}\Bigr)=(-1)^{k+1}\Cot(k,1,4k).
\]
Since $\zeta_{4k}^{kj}=i^j$, the left side equals $\sum_{j=1}^{4k-1} i^j\tan(\pi j/(4k))$.

Using Theorem~\ref{thm:cot-odd} with $n=0$, $a=k$, $k\to4k$ yields $\Cot(k,1,4k)=2ik$, so
\[
\Tan(k,1,4k)=(-1)^{k+1}\cdot 2ik.
\]

The symmetry $\tan\bigl(\frac{\pi(4k-j)}{4k}\bigr)=-\tan\bigl(\frac{\pi j}{4k}\bigr)$ and $i^{4k-j}=i^{-j}=-i^j$ (for odd $j$) show that terms at $j$ and $4k-j$ coincide. Hence
\[
\sum_{\substack{j=1\\ j\ \text{odd}}}^{2k-1} i^j \tan\Bigl(\frac{\pi j}{4k}\Bigr)=\frac12\Tan(k,1,4k)=(-1)^{k+1}ik.
\]
Now set $j=2m+1$ ($m=0,\dots,k-1$). Then $i^j=i(-1)^m$, and
\[
\sum_{m=0}^{k-1} (-1)^m \tan\Bigl(\frac{\pi(2m+1)}{4k}\Bigr)
= i^{-1}\sum_{\substack{j=1\\ j\ \text{odd}}}^{2k-1} i^j \tan\Bigl(\frac{\pi j}{4k}\Bigr)
= -i\cdot (-1)^{k+1}ik = (-1)^{k+1}k.
\]
This concludes the proof.
\end{proof}

Finally, we provide a concrete computational example to demonstrate the efficiency of our derived formulas.

\noindent\textbf{Example.}
Take $k=5$, $a=3$, $n=3$ in formula \ref{eq:Tan-odd}. Since $k$ is odd, no singular index $j_0$ exists. Let $\zeta=e^{2\pi i/5}$. We aim to evaluate
\[
\Tan(3,3,5)=\sum_{j=1}^{4}\zeta^{3j}\tan^{3}\!\Bigl(\frac{\pi j}{5}\Bigr).
\]
Using the precomputed coefficients $E_2(\tfrac{3}{5})=-\tfrac{6}{25}$, $E_0(x)=1$, $r_{3,0}=1$, $r_{3,1}=\tfrac14$, we compute:
\begin{align*}
\Tan(3,3,5)
&= -i\cdot 2^{2}\Bigl(\frac{E_{2}(\tfrac{3}{5})}{2!}\cdot 5^{3}+\frac{1}{4}\cdot\frac{E_{0}(\tfrac{3}{5})}{0!}\cdot 5\Bigr) \\
&= -4i\Bigl(-15+\frac{5}{4}\Bigr) = 55i.
\end{align*}
We thus obtain the closed-form value
\[
\sum_{j=1}^{4}\zeta^{3j}\tan^{3}\!\Bigl(\frac{\pi j}{5}\Bigr)=55i.
\]

\section{Concluding Remarks}
\label{sec:conclusion}

We have applied the constant term method for iterated Laurent series to finite trigonometric
power sums of cotangent, tangent, cosecant and secant weighted by primitive $k$th roots of unity.
In contrast to conventional approaches that rely on convolutions or fractional part functions,
our unified framework yields explicit closed forms for both even and odd powers by combining
generating functions, partial fraction decomposition and constant term extraction.

Using Euler’s formula, we converted the trigonometric sums into evaluations of rational functions
and then reduced the summations to the extraction of constant terms from iterated Laurent series.
This strategy avoids cumbersome case distinctions and significantly simplifies the derivations.
As applications we recovered numerous classical identities—for instance, the ordinary and
alternating cotangent power sums and Acton’s alternating tangent sum.

The results highlight the deep interplay between cyclotomic roots, Bernoulli and Euler
polynomials, and finite trigonometric sums, and they provide effective tools for problems in
analytic number theory.  The constant term method can be further extended to twisted sums,
lattice point enumeration and topics related to modular forms.

A natural question left open by our work is the existence of analogous closed-form expressions
for \emph{odd-power} sums of cosecant and secant with root-of-unity weights.  Unlike the
even-power case, the elementary identities for odd powers involve additional derivatives or mixed
products, and a direct binomial expansion is no longer sufficient.  While partial results can be
obtained by the techniques of this paper, a complete explicit formula remains an open problem.

\vspace{6pt}
\noindent\textbf{Acknowledgments.}
This research was partially supported by the National Natural Science Foundation of China
under Grant No.\ 12571355.



\bibliographystyle{plain}

\end{document}